\documentclass{elsarticle}
\usepackage{amsmath,amssymb}
\newtheorem{thm}{Theorem}[section]
\newtheorem{cor}[thm]{Corollary}
\newtheorem{lem}[thm]{Lemma}

\newtheorem{conj}[thm]{Conjecture}

\newdefinition{defn}{Definition}

\newproof{proof}{Proof}
\newproof{proof1}{Proof of Theorem~\ref{thm:allbetterlower}}

\newcommand{\Ff}{\mathcal F}
\newcommand{\Ee}{\mathcal E}
\newcommand{\Rr}{\mathcal R}
\newcommand{\Oo}{\mathcal O}
\newcommand{\Ii}{\mathcal I}
\newcommand{\sd}{\bigtriangleup}

\newcommand{\dcon}{/ \!\! /}
\newcommand{\ddel}{\backslash \!\! \backslash}

\DeclareMathOperator{\cl}{cl}

\begin{document}

\title{How many delta-matroids are there?}

\author[a1]{Daryl Funk\fnref{fn1}}
\ead{daryl.funk@vuw.ac.nz}
\author[a1]{Dillon Mayhew\fnref{fn1}}
\ead{dillon.mayhew@vuw.ac.nz}
\author[a2]{Steven D. Noble\corref{cor1}}
\ead{s.noble@bbk.ac.uk}
\cortext[cor1]{Corresponding author}
\fntext[fn1]{Supported by the Rutherford Discovery Fellowship.}
\address[a1]{School of Mathematics and Statistics, Victoria University of Wellington,
PO Box 600, Wellington 6140, New Zealand}
\address[a2]{Department of Economic, Mathematics and Statistics, Birkbeck, University of London, Malet St, London, WC1E 7HX United Kingdom}

\begin{abstract}
We give upper and lower bounds on the number of delta-matroids, and on the number of even delta-matroids.
\end{abstract}


\date{\today}

\maketitle

\section{Introduction}

Matroids are important combinatorial structures, introduced in 1935 by Whitney~\cite{Whitney} as a combinatorial abstraction of the properties of linear independence. They arise in graph theory, linear algebra, transversal theory and have been widely studied partly due to their connection with combinatorial optimization and particularly the greedy algorithm. A matroid comprises a pair $(E,\mathcal I)$, where $E$ is a finite set called the \emph{ground set} and $\mathcal I$ is a non-empty collection of subsets called \emph{independent sets}, satisfying the following two conditions.

\begin{enumerate}
\item If $I_1 \in \Ii$ and $I_2 \subseteq I_1$, then $I_2 \in \Ii$.
\item If $I_1, I_2 \in \Ii$ and $|I_2| > |I_1|$, then there exists an element $x$ of $I_2-I_1$ such that $I_1 \cup x \in \Ii$.
\end{enumerate}

An intriguing problem has been to determine good bounds on the number $m_n$ of labelled matroids with ground set $\{1,\ldots,n\}$. The first non-trivial upper bound was proved by Piff~\cite{Piff}, who showed that
\[ \log \log m_n \leq n -\log n + O(\log \log n).\]
(In this paper $\log$ denotes logarithms taken to base two.)
Only a year later Knuth~\cite{Knuth} showed that
\[ \log \log m_n \geq n- \frac 32 \log n - O(1).\]
Little progress was made until recently Bansal, Pendavingh and van der Pol~\cite{Ban+Pen+vdP:num-matroids} made a significant advance by proving that
\[ n- \frac 32 \log n  + \frac 12 \log \frac 2 {\pi} -o(1) \leq \log \log m_n \leq  n- \frac 32 \log n  + \frac 12 \log \frac 2 {\pi} +1 + o(1).\]

Delta-matroids are a generalization of matroids introduced by Bouchet~\cite{Bou:Symm} and extensively studied, primarily by Bouchet (e.g.~\cite{Bou:Map,Bou+Duch}), in the late 1980s. They arise in the theory of embedded graphs, linear algebra and in the structure of Eulerian tours in four-regular graphs. Recently they have attracted more attention due to the work of Brijder and Hoogeboom, Chun, Moffatt, Noble and Rueckriemen, and Traldi. See for example~\cite{Bri+Hoog:1,Bri+Hoog:2,CMNR1,CMNR2,Tralid+Brijder+Hoogeboom}.

A \emph{delta-matroid} $(E,\Rr)$ comprises a finite ground set and a non-empty collection of subsets of $E$ satisfying the
\emph{symmetric exchange axiom}:
\begin{quote}
For every pair $X,Y \in \Ff$, if $e \in X \sd Y$ then there exists $f \in X \sd Y$ so that $X \sd \{e,f\} \in \Ff$.
\end{quote}
(Note that $e=f$ is permitted.)
The sets in $\Ff$ are the \emph{feasible} sets of the delta-matroid.

The maximal independent sets of a matroid are called \emph{bases}.
It is not difficult to show that the bases of a matroid form the feasible sets of a delta-matroid with the same ground set (for instance, by combining Lemmas 1.2.2 and 2.1.2 of~\cite{Oxley}). The feasible sets of a delta-matroid may differ in size, but if the feasible sets of a delta-matroid all have the same size then they form the bases of a matroid.

We prove the following bounds on the number $d_n$ of labelled delta-matroids with ground set $\{1,\ldots,n\}$.
\begin{thm}\label{thm:allmain}
$n-1 < \log\log (d_n + 1) \leq n-1 + 0.369$.
\end{thm}

\begin{thm}\label{thm:allbetterlower}
For any $\epsilon > 0$ and all sufficiently large $n$, $d_n \geq  (1 - \epsilon) n 2^{2^{n-1}}$.
\end{thm}

These results indicate that there are many more delta-matroids than there are matroids.
A delta-matroid in which the sizes of the feasible sets all have the same parity is called \emph{even}. Our third result gives bounds on the number $e_n$ of labelled delta-matroids with ground set $\{1,\ldots,n\}$ which are more reminiscent of the bounds on $m_n$.

\begin{thm}\label{thm:even}
$n - \log n - 1 \leq \log\log e_n \leq n - \log n + O(\log\log n)$.
\end{thm}

\section{Preliminaries}

We assume familiarity with the basic theory of matroids and refer the reader to the monograph by Oxley~\cite{Oxley}. Given a matroid $M$, we use $r_M$ to denote its rank function and $\cl_M$ to denote its closure operator, omitting $M$ when the context is clear. We use $[n]$ to denote the set $\{1,\ldots,n\}$.

A \emph{set system} is a pair $(E,\Ff)$, where $E$ is a finite ground set and $\Ff$ is a collection of subsets of $E$. If $\Ff$ is non-empty, we say that $(E,\Ff)$ is \emph{proper}; otherwise it is \emph{improper}.
We define two operations on set systems, namely deletion and contraction.
Let $S=(E,\Ff)$ be a set system and let $e\in E$.
Then $S \ddel e$, the \emph{deletion} of $e$ from $S$, is the set system
$(E-e, \{F \in \Ff : e \notin F\})$; on the other hand $S \dcon e$, the \emph{contraction} of $e$ from $S$, is the set system
$(E-e,\{F-e: F \in \Ff\text{ and } e\in F\}))$.

Bouchet and Duchamp~\cite{Bou+Duch} defined the operations of deletion and contraction on a delta-matroid. These operations are similar to, but not exactly the same as the deletion and contraction operations that we defined on set systems. They differ in the way in which they treat the contraction of an element that does not appear in any feasible set and the deletion of an element that appears in every feasible set. Nevertheless, for our purposes, it is the operations on set systems defined earlier that we need to apply to delta-matroids. If we contract an element that does not appear in any feasible set or delete an element that appears in every feasible set the resulting set system is improper and consequently not a delta-matroid. In all other cases it is not difficult to show directly by applying the definition of a delta-matroid that the result of contracting or deleting an element from a delta-matroid is a delta-matroid. Because of the slight difference from standard practice, we use the double slash notation.

A fundamental operation on delta-matroids, introduced by Bouchet in~\cite{Bou:Symm}, is the twist.
Let $D=(E,\Ff)$ be a delta-matroid and $A$ be a subset of $E$.
The \emph{twist} of $D$ with respect to $A$, written $D*A$, is the delta-matroid $(E, \{A \sd X : X \in \Ff\})$. (It is easy to see that $D*A$ genuinely is a delta-matroid.)
The \emph{dual} of $D$ is $D^* = D*E$.

\section{How many delta-matroids are there?}

In this section we prove Theorems~\ref{thm:allmain} and~\ref{thm:allbetterlower} by giving upper and lower bounds on the number $d_n$ of labelled delta-matroids with ground set $[n]$.
The $n$-dimensional hypercube $Q_n$ is the graph on vertex set $\{0,1\}^n$ in which two vertices are adjacent if they differ in exactly one coordinate.
Consider each vertex as a 0,1 indicator vector: in this way the vertices of $Q_n$ are in one-to-one correspondence with the subsets of $[n]$. To aid exposition, we will sometimes conflate subsets of $[n]$ and vertices of $Q_n$.
We say that a vertex of $Q_n$ has \emph{even support} if its corresponding indicator vector has an even number of ones; otherwise we say that it has \emph{odd support}.
The hypercube $Q_n$ is $n$-regular and bipartite with parts $\Ee$ and $\Oo$, where $\Ee$ is the set of all vertices with even support.

We begin by establishing a lower bound on $d_n$.
\begin{lem} \label{lem:deltamatroidfromhypercube}
The complement of a stable set in $Q_n$ corresponds to the family of feasible sets of a delta-matroid.
\end{lem}

\begin{proof}
Let $I$ be a stable set of vertices in $Q_n$, and let $\Ff = V(Q_n) \setminus I$.
Let $X, Y \in \Ff$ and let $e \in X \sd Y$.
If $X \sd Y = \{e\}$, then $X \sd e = Y \in \Ff$.
So assume $|X \sd Y|>1$.
If $X \sd e \in \Ff$ we are done, so suppose not.
Then $X \sd e \in I$, and all neighbours of $X \sd e$ in $Q_n$ are in $\Ff$.
Let $f \in X \sd Y - e$.
Since $X \sd \{e,f\}$ is a neighbour of $X \sd e$ in $Q_n$ and $X \sd \{e,f\} \notin I$, we have $X \sd \{e,f\} \in \Ff$.\qed
\end{proof}

\begin{cor} \label{cor:deltamatroidfromevensets}
Let $A$ be an arbitrary collection of subsets of $[n]$ of even cardinality, and let $O$ be the collection of all subsets of $[n]$ of odd cardinality.
Then $A \cup O$ is the collection of feasible sets of a delta-matroid.
\end{cor}

\begin{proof}
The elements of $A \cup O$ correspond to the complement of a stable set in $Q_{n}$.\qed
\end{proof}

The next corollary follows immediately and establishes the lower bound in Theorem~\ref{thm:allmain}.
\begin{cor} \label{cor:lowerbound1}
$d_n \geq 2^{2^{n-1}}$.
\end{cor}

\begin{proof}
The number of subsets of even cardinality of a ground set of size $n$ is $2^{n-1}$, so the bound follows from Corollary~\ref{cor:deltamatroidfromevensets}.\qed
\end{proof}

This bound can be improved by using the following result due to Korshunov and Sapozhenko~\cite{Kor+Sap:hypercube}.
\begin{thm}[Korshunov and Sapozhenko]
The number of stable sets in $Q_n$ is $2\sqrt e(1+o(1))2^{2^{n-1}}$.
\end{thm}

\begin{cor}
$d_n \geq 2 \sqrt e 2^{2^{n-1}}$.
\end{cor}

The class of delta-matroids arising from the complement of a stable set in the hypercube perhaps forms the natural delta-matroid analogue of the class of sparse paving matroids, which we will need in the next section and now define.
A matroid is \emph{paving} if it has no circuits of size strictly smaller than its rank. It is \emph{sparse paving} if both it and its dual are paving. It is not difficult to show that a matroid $M$ is sparse paving if and only if every subset of $E(M)$ having size $r(M)$ is either a basis or a circuit--hyperplane. Moreover every hyperplane of a sparse paving matroid $M$ has size $r(M)$ or $r(M)-1$.
Welsh~\cite{Welsh} asked whether or not most matroids are paving and later Mayhew, Newman, Welsh and Whittle~\cite{Mayhew+Newman+Welsh+Whittle:asymptoticmatroid} conjectured that asymptotically almost all matroids are paving, which would imply that asymptotically almost all matroids are sparse paving. Theorem~\ref{thm:allbetterlower} implies that, in contrast, the class of delta-matroids arising from the complement of a stable set in the hypercube forms a vanishingly small proportion of the class of all delta-matroids.

To prove Theorem~\ref{thm:allbetterlower} we use a strengthening of Lemma~\ref{lem:deltamatroidfromhypercube}.
\begin{lem}\label{lem:matching}
Let $n\geq 2$ and let $S$ denote a subset of the vertices of $Q_n$ such that the induced subgraph $Q_n[S]$ has maximum degree one. Then the complement of $S$ forms the set of feasible sets of a delta-matroid.
\end{lem}

\begin{proof}
Let $\Ff = V(Q_n) \setminus S$ and let $X,Y \in \Ff$. We may assume that $|X \sd Y| \geq 3$ else there is nothing to prove. Let $e\in X \sd Y$.
If $X \sd e \in \Ff$ we are done, so suppose not.
Then $X \sd e \in S$ and at most one neighbour of $X\sd e$ in $Q_n$ is in $S$. So every other neighbour is in $\Ff$.
Let $f \in (X \sd Y)-e$. Then there is at most one choice for $f$ such that $X\sd \{e,f\} \notin \Ff$. Therefore there are at least $|X \sd Y|-2 \geq 1$ choices for $f$ such that $X\sd \{e,f\} \in \Ff$.\qed
\end{proof}

We now prove Theorem~\ref{thm:allbetterlower}, establishing a better lower bound for $d_n$.
\begin{proof1}
Choose one of the $n$ edge cuts of $Q_n$ that separates $Q_n$ into two copies of $Q_{n-1}$.
Let us denote these two copies by $Q_{n-1}^e$ and $Q_{n-1}^o$.
Let $A_e$ denote the random subset of vertices of $Q_{n-1}^e$ with even support obtained by choosing each one independently with probability $1/2$, and let $A_o$ be the similarly defined random subset of
vertices of $Q_{n-1}^o$ with odd support. Then $A_e$ is a stable set in $Q_{n-1}^e$ and $A_o$ is a stable set in $Q_{n-1}^o$. So every component of the subgraph of $Q_n$ induced by $A_e \cup A_o$ is either an isolated vertex or an edge of the cut separating $Q_n$ into $Q_{n-1}^e$ and $Q_{n-1}^o$. By applying Lemma~\ref{lem:matching} one can show that the complement of $A_e \cup A_o$ corresponds to the collection of feasible sets of a delta-matroid.

The $n$ edge cuts separating $Q_n$ into two copies of $Q_{n-1}$ are pairwise disjoint. Therefore as long as the subgraph of $Q_n$ induced by $A_e \cup A_o$ contains at least one edge, the set $A_e \cup A_o$ cannot be chosen when starting with a different choice from amongst the $n$ edge cuts. Hence, as long as we always have such an edge, no double counting will occur in the following count of the number of such choices.
The maximum possible number of edges in the subgraph of $Q_n$ induced by $A_e \cup A_o$ is $2^{n-2}$ and each of these is absent independently with probability $3/4$.
So the probability that no such edge is induced is $(3/4)^{2^{n-2}}$.

Therefore the number of delta-matroids produced in this way is
\[
n \cdot 2^{2^{n-2}} \cdot 2^{2^{n-2}} \cdot \big(1 - (3/4)^{2^{n-2}}\big).
\]\qed
\end{proof1}
We have not tried hard to find a better lower bound on the number of induced subgraphs of $Q_n$ with maximum degree one, so it may be simple to improve this bound. As far as we know, there is no relevant previous work.

We now move on to establishing upper bounds for $d_n$.

\begin{thm} \label{thm:upperandlowerbounds}
The sequence $\Gamma_n = \log \log (d_n+1) - (n-1)$ is strictly positive and decreasing for $n \geq 2$.
\end{thm}

\begin{proof}
Corollary~\ref{cor:lowerbound1} implies that $\Gamma_n$ is strictly positive.
Clearly $d_n+1$ counts the number of set systems on $n$ elements that are either improper or form a delta-matroid.
Notice that there is a one-to-one correspondence between set systems with ground set $[n+1]$ and pairs of set systems with ground set $[n]$ given by the mapping $S \mapsto (S \ddel n+1, S \dcon  n+1)$.
Moreover if the set system $S$ is either a delta-matroid or improper, then both $S \ddel n+1$ and $S\dcon  n+1$ are either delta-matroids or empty. Consequently $d_{n+1}+1 \leq (d_n+1)^2$.
Observe that for $n\geq 2$ the set system $([n+1],\{\emptyset, [n+1]\})$ is not a delta-matroid, but that both $([n+1],\{\emptyset, [n+1]\}) \ddel n+1 = ([n],\{\emptyset\})$ and
$([n+1],\{\emptyset, [n+1]\}) \dcon n+1 = ([n],\{[n]\})$
are delta-matroids.
Hence $d_{n+1}+1 < (d_n+1)^2$, and the fact that $\Gamma_n$ is strictly decreasing follows by taking logs twice.\qed
\end{proof}

The following corollary is immediate.
\begin{cor} \label{cor:upperbounds}
For positive integers $n$ and $k$ with $n\geq k$, \[\log \log (d_n+1) \leq n + \log \log (d_k + 1) - k. \]
\end{cor}

Counting delta-matroids by computer, we obtain $d_1 = 3$, $d_2=15$, $d_3=155$, $d_4=5959$, $d_5=4 980 259$ and $d_6=2 746 801 811 279$. The code used is available from \texttt{http://eprints.bbk.ac.uk/id/eprint/19837} and the numbers have been independently verified by Royle~\cite{Royle}. Briefly, for $n\leq 5$ a list of all labelled delta-matroids
with ground set $[n]$ is computed by running through all ordered pairs
$(D_1,D_2)$ of labelled delta-matroids with ground set $[n-1]$ and checking whether the set system $D$
with ground set $[n]$ such that $D \dcon n=D_1$ and $D\ddel n=D_2$ is a delta-matroid.
If $D$ is a delta-matroid then we say that $D_1$ and $D_2$ are \emph{compatible}.
A proper set system $D=(E,\mathcal F)$ is a delta-matroid if and only if for all $e \in E$, both $D\ddel e$ and $D\dcon e$ are delta-matroids, and $D$ has no antipodal pair of feasible sets that violate the symmetric exchange axiom.
Bonin, Chun and Noble~\cite{BCN} have shown that if $|E|\geq 5$, then a set system $D$ that is not
a delta-matroid but both $D\ddel e$ and $D\dcon e$ are delta-matroids for all $e\in E$ must have set of feasible sets comprising a pair of antipodal sets. Thus for $n=5$, the code runs through all pairs $(D_1,D_2)$ of labelled delta-matroids with ground set $\{1,2,3,4\}$, forms the set system $D$ as described above, checks whether each single element deletion and contraction belongs to the list of labelled delta-matroids with ground set $\{1,2,3,4\}$ and finally checks $D$ against the list of 16 set systems with a ground set of five elements and set of feasible sets comprising two antipodal sets.

Define an equivalence relation on labelled delta-matroids so that two labelled delta-matroids are equivalent if one is isomorphic to a twist of the other. For $n=6$, the number of potential delta-matroids is too large to allow the method used for $n=5$
to work in a reasonable period of time, so a unique representative from each equivalence class is used in the role of $D_1$. The number of labelled delta-matroids $D_2$ such that $D_1$ and $D_2$ are compatible is independent of the choice of $D_1$ from its equivalence class. For each
representative $D_1$ of an equivalence class, the number of delta-matroids $D_2$ such that $D_1$ and $D_2$ are compatible is computed exhaustively in the same way as for $n=5$ and is multiplied by the size of the equivalence class of $D_1$. Finally these numbers are summed as $D_1$ ranges over representatives from the equivalence classes.

The corresponding values of $\Gamma_n$ are
$\Gamma_1=\Gamma_2=1$, $\Gamma_3 \simeq 0.865$, $\Gamma_4 \simeq 0.649$, $\Gamma_5 \simeq 0.476$, $\Gamma_6 \simeq 0.369$.
Thus, by applying the previous corollary, we obtain the following, completing the proof of Theorem~\ref{thm:allmain}.
\begin{cor}
	$\log\log(d_n+1) \leq n + \log\log(d_6 + 1) - 6 \leq n-1 + 0.369$.
\end{cor}

Since the sequence $(\Gamma_n)_{n\geq2}$ is decreasing and bounded below by zero, the limit $\lim_{n \to \infty} \Gamma_n$ exists.
Given the speed with which $\Gamma_n$ is decreasing and our inability to find larger classes of delta-matroids than those constructed in the proof of Theorem~\ref{thm:allbetterlower}, we make the following conjecture.
\begin{conj}
	$\Gamma_n \to 0$ as $n \to \infty$.
\end{conj}

\section{How many even delta-matroids are there?}

Recall that $e_n$ denotes the number of labelled even delta-matroids with ground set $[n]$.
We first describe a construction from which a large number of even delta-matroids arise.
The Johnson graph $J(n,r)$ has vertices corresponding to all the subsets of $[n]$ having size $r$, with two vertices joined by an edge if the intersection of the corresponding subsets has size $r-1$.
As noted by Bansal, Pendavingh and van der Pol~\cite{Ban+Pen+vdP:num-matroids}, who include a proof, Piff and Welsh~\cite{Piff+Welsh} essentially showed that
a collection of subsets of $[n]$ each with size $r$, for some $r$ safisfying $0<r<n$, is the collection of circuit--hyperplanes of a sparse paving matroid if and only if it corresponds to a stable set in $J(n,r)$.
Furthermore it was shown by Graham and Sloane~\cite{Graham+Sloane} that $J(n,r)$ contains a stable set of size at least $\frac 1n \binom nr$.

Choose a collection $\mathcal F$ of even-sized subsets of $[n]$ so that for all $r$ satisfying $0 \leq r \leq \lfloor n/2\rfloor$ the subsets of size $2r$ are the bases of a sparse paving matroid with ground set $[n]$ and rank $2r$.

\begin{lem} \label{lem:constructingevendeltamatroids}
	$\mathcal F$ is the collection of feasible sets of a delta-matroid.
\end{lem}

\begin{proof}
	Choose $F_1, F_2 \in \Ff$. For $i=1,2$, denote by $M_i$ the sparse paving matroid for which the bases are the elements of $\Ff$ having size $|F_i|$. If $|F_1| = |F_2|$, then the symmetric exchange axiom holds because, by construction, the collection of all elements of $\Ff$
having a common size forms the collection of bases of a matroid and the basis exchange axiom holds for such a collection.

So suppose $|F_1| < |F_2|$. Let $e \in F_1 \sd F_2$. Suppose first that $e \in F_1$.
	Since $\cl_{M_1}(F_1 - e)$ is a hyperplane of $M_1$, we have  $|\cl_{M_1}(F_1 - e)|\leq |F_1|$. Furthermore, because $|F_2| \geq |F_1| + 2$, there is an element $f \in F_2 - F_1$ with $f \notin \cl_{M_1}(F_1 - e)$.
	Hence $F_1 \sd \{e,f\}$ is a basis of $M_1$ and belongs to $\Ff$.

	Now suppose $e \in F_2$.
	If $|F_2-F_1| \leq 2$, then $F_1 \subseteq F_2$ and $|F_2|=|F_1|+2$; clearly there is an element $f$ such that $F_1 \sd \{e,f\} = F_2$ and we are done. Consequently we may assume that $|F_2-F_1| \geq 3$.
    Let $M_3$ denote the sparse paving matroid in the construction of $\Ff$ with rank $|F_1| + 2$.
    Then $F_1 \cup e$ is independent in $M_3$.
	So $\cl_{M_3}(F_1 \cup e)$ is a hyperplane in $M_3$ and $|\cl_{M_3}(F_1 \cup e)| \leq |F_1|+2$. So there is an element $f \in F_2 - F_1$ such that $f \notin \cl_{M_3}(F_1 \cup e)$. Hence $F_1 \sd \{e,f\}$ is a basis of $M_3$ and belongs to $\Ff(D)$.
	
	Finally suppose $|F_1| > |F_2|$.
	Consider $E-F_1$ and $E-F_2$ as bases of $M_1^*$ and $M_2^*$, respectively.
	These are both sparse paving matroids.
	Let $e \in F_1 \sd F_2 = (E - F_1) \sd (E-F_2)$.
	The previous argument shows that there is an element $f \in (E - F_1) \sd (E-F_2)$ such that $(E-F_1) \sd \{e,f\}$ is a basis of either $M_1^*$ or $M_3^*$, where $M_3$ is as defined in the previous paragraph.
	Hence $f \in F_1 \sd F_2$ and $E-((E-F_1) \sd \{e,f\}) = F_1 \sd \{e,f\}$ is a basis of $M_1$ or $M_3$ and consequently a member of $\Ff(D)$.\qed
\end{proof}

We now establish the lower bound in Theorem~\ref{thm:even}.
\begin{thm} The number of even delta-matroids $e_n$ satisfies
\[
\log\log e_n \geq n - 1 - \log n.
\]
\end{thm}

\begin{proof}
First note that the bound holds when $n \leq 2$, 
so we may assume $n \geq 3$.
Let $f_n$ denote the number of delta-matroids of the form of Lemma \ref{lem:constructingevendeltamatroids}.
Then $e_n \geq f_n$. If $0<r<n$, it follows from the discussion above that the number of labelled sparse paving matroids with ground set $[n]$ and rank $r$ is equal to the number of stable sets of $J(n,r)$. Since $J(n,r)$ has a stable set of size at least $\frac 1n \binom nr$, it has at least $2^{\frac 1n \binom nr}$ stable sets.
To accommodate the cases $r=0$ and $r=n$, we proceed as follows.

Suppose first that $n$ is even and consequently $n \geq 4$.
Then $J(n,2)$ has a stable set $\{\{1,2\},\{3,4\},\ldots,\{n-1,n\}\}$ of size $n/2$ and consequently, at least $2^{n/2}$ stable sets.
Since $n \geq 4$,
\[ 2^{n/2} \geq 2^{(n-1)/2} 2^{2/n} = 2^{\frac{1}{n} \cdot \binom n{0}} 2^{\frac{1}{n} \cdot \binom n{2}} 2^{\frac{1}{n} \cdot \binom n{n}}\]
and we have
\[ f_n \geq 2^{n/2} \prod_{r=2}^{n/2-1} 2^{\frac{1}{n} \cdot \binom n{2r}} \geq \prod_{r=0}^{n/2} 2^{\frac{1}{n} \cdot \binom n{2r}}
= 2^{\sum_{r=0}^{n/2} \frac{1}{n} \cdot \binom n{2r}} = 2^{\frac{1}{n} \cdot 2^{n-1}}\]
as required.

Now suppose that $n$ is odd. Then $J(n,2)$ has stable sets $S_1=\{\{1,2\}$, $\{3,4\}$, \ldots, $\{n-2,n-1\}\}$ and $S_2=\{\{2,3\}$, $\{4,5\}$, \ldots $\{n-1,n\}\}$ each of size $(n-1)/2$. Consequently it has at least $2 \cdot 2^{(n-1)/2}-1$ stable sets, as the only common subset of $S_1$ and $S_2$ is the empty set.
Therefore $J(n,2)$ has at least $2^{n/2}$ stable sets.
Since $n \geq 3$,
\[ 2^{n/2} \geq 2^{(n-1)/2} 2^{1/n} = 2^{\frac{1}{n} \cdot \binom n{0}} 2^{\frac{1}{n} \cdot \binom n{2}}\]
and we have
\[ f_n \geq 2^{n/2} \prod_{r=2}^{(n-1)/2} 2^{\frac{1}{n} \cdot \binom n{2r}} \geq \prod_{r=0}^{(n-1)/2} 2^{\frac{1}{n} \cdot \binom n{2r}}
= 2^{\sum_{r=0}^{(n-1)/2} \frac{1}{n} \cdot \binom n{2r}} = 2^{\frac{1}{n} \cdot 2^{n-1}}\]
as required.\qed
\end{proof}

To obtain an upper bound on the number of even delta-matroids, we use a similar procedure to that in \cite{Ban+Pen+vdP:num-matroids}, where a bounded-size stable set in a Johnson graph together with a carefully chosen collection of flats is used to encode a matroid.

We will assume for now that our delta-matroids only have feasible sets of even cardinality. The map $D \mapsto D*\{1\}$ gives a one-to-one correspondence from delta-matroids with ground set $[n]$ in which all feasible sets have even cardinality to those in which all feasible sets have odd cardinality, so the number of delta-matroids having only feasible sets of even cardinality is half the total number of even delta-matroids.

Let $R_n$ be the graph with vertex set $V(Q_n)$ in which two vertices are adjacent if and only if they are at distance 2 in $Q_n$.
The graph $R_n$ is regular of degree $\binom n2$ and has two isomorphic connected components, whose vertex sets correspond to the subsets of $[n]$ of even and odd support, respectively.

Let $D=(E,\Ff)$ be a delta-matroid in which all feasible sets have even cardinality, and
let $L$ denote the vertices of $R_n$ that have even support that correspond to infeasible sets of $D$. In order to provide an upper bound on the number of even delta-matroids, our aim is to provide a short description of $L$ and then to bound the total number of possible descriptions. There are two key elements to this. First we apply an encoding procedure due to Bansal, Pendavingh and van der Pol~\cite{Ban+Pen+vdP:num-matroids} that takes an arbitrary set $L$ of vertices in a graph $G$ and finds a pair $(S,A)$ of sufficiently small sets satisfying $S \subseteq L \subseteq S \cup N(S) \cup A$, where $N(S)$ is the set of vertices of $G$ that are neighbours of some vertex of $S$. The authors of~\cite{Ban+Pen+vdP:num-matroids} adapted it from work of Alon, B\'alogh, Morris and Samotij~\cite{Alon+Balogh+Morris+Samotij}, who themselves credit Kleitman and Winston~\cite{Kleitman+Winston} with the original idea.

We describe briefly how the procedure works, following~\cite{Ban+Pen+vdP:num-matroids}, where full details and proofs are given. It takes as input a graph $G=(V,E)$ and a subset $L$ of $V$ and outputs a pair $(S,A)$ of subsets of $V$. We assume that $V$ is given a fixed ordering, purely to break ties in the procedure.
Initially $S$ is empty and $A=V$. As the procedure runs, $S$ increases in size and $A$ decreases. The procedure stops when $|A| \leq \alpha |V|$, where $\alpha$ will be specified later. At each stage a vertex $v$ of $A$ with maximum degree in the induced subgraph $G[A]$ is chosen, with ties broken according to the ordering of $V$. If $v \notin L$ then $v$ is removed from $A$ and the procedure moves onto another stage. If $v\in L$, then $v$ and all of its neighbours in $G[A]$ are removed from $A$ and $v$ is added to $S$.

The following lemma, originally from~\cite{Alon+Balogh+Morris+Samotij} and restated in~\cite{Ban+Pen+vdP:num-matroids}, is crucial.
\begin{lem}
When the procedure terminates, the set $A$ is completely determined by $S$, irrespective of $L$.
\end{lem}

The following lemma is from~\cite{Ban+Pen+vdP:num-matroids}.
\begin{lem}
Suppose that $G$ has $N$ vertices, is $d$-regular and the smallest eigenvalue of its adjacency matrix is $-\lambda$. Let $\alpha =  \frac{\lambda}{d+\lambda}$. Then at the end of the procedure described above, we have $|S| \leq \big\lceil \frac{\ln(d+1)}{d+\lambda}N\big\rceil$.
\end{lem}

It is not difficult to find the smallest eigenvalue of the adjacency matrix of a connected component of $R_n$.
\begin{lem}
The smallest eigenvalue of the adjacency matrix of a connected component of $R_n$ is $-n/2$ if $n$ is even, and $(1-n)/2$ if $n$ is odd.
\end{lem}

\begin{proof}
Denote the adjacency matrix of a graph $G$ by $A(G)$.
Whenever $u$ and $v$ have a common neighbour in $Q_n$, they have exactly 2 common neighbours, so
\[
A(R_n) = \frac{(A(Q_n))^2 - nI_n}{2}
\]
Therefore if $v$ is an eigenvector of $A(Q_n)$ with eigenvalue $\lambda$, then
\[A(R_n) v 	= \frac{1}{2} (A(Q_n))^2 v - \frac{n}{2} I_n v = \frac{\lambda^2 }{2}v - \frac{n}{2} v  = \frac{(\lambda^2-n)
}{2} v,
\]
so $v$ is an eigenvector of $A(R_n)$ with eigenvalue $(\lambda^2-n)/2$. The matrix $A(Q_n)$ is symmetric, so there is a basis $B$ of ${\mathbb R}^{2^n}$ comprising eigenvectors of $A(Q_n)$. We have just shown that all of the vectors in $B$ are also eigenvectors of $A(R_n)$, so every eigenvalue of $A(R_n)$ must be associated with an eigenvector that is also an eigenvector of $A(Q_n)$. Thus $\lambda'$ is an eigenvalue of $R_n$ if and only if $\lambda'=(\lambda^2-n)/2$ where $\lambda$ is an eigenvalue of $A(Q_n)$.

The eigenvalues of $Q_n$ are $-n, -n+2, \ldots, n-2, n$ \cite[p.\ 10]{Brouwer+Haemers:spectra}.
Hence $R_n$ has eigenvalues (listed with multiplicities) $\frac{(-n)^2-n}{2}$, $\frac{(-n+2)^2-n}{2}$, \ldots, $\frac{n^2-n}{2}$.
The result follows as the two components of $R_n$ are isomorphic.\qed
\end{proof}

The second key requirement of the proof is for an even sized infeasible set $X$ to describe concisely which sets of the form $X \sd \{e,f\}$ are infeasible. In other words suppose that $x$ is a vertex of $R_n$ corresponding to an even sized infeasible set $X$, then we wish to describe concisely which neighbours of $x$ in $R_n$ correspond to infeasible sets. Such a description will be used for each vertex of $S$ in the encoding procedure in order to specify the vertices of $S \cup N(S)$ corresponding to infeasible sets.

\begin{lem}
Let $D=(E,\Ff)$ be a delta-matroid and let $X$ be an infeasible set of $D$. Let $\mathcal B$ denote the collection of sets $Y$ with the smallest size possible such that $X \sd Y \in \Ff$. Then $\mathcal B$ forms the collection of bases of a matroid with ground set $E$.
\end{lem}

\begin{proof}
Bouchet~\cite{Bou:Map} proved that the collection of feasible sets of a delta-matroid with minimum cardinality form the bases of a matroid. Now $\mathcal B$ is the collection of feasible sets of the delta-matroid $D*X$ having minimum size and consequently forms the bases of a matroid with ground set $E$.\qed
\end{proof}

Notice that there is a one-to-one correspondence betweens matroids with rank two on ground set $E$ and partitions of $E \cup z$ with at least three blocks, where $z$ is an arbitrary element not contained in $E$. The partition corresponding to a matroid $M$ is formed by taking one block to comprise all the loops of $M$ together with $z$ and each other block to be a parallel class of non-loop elements. In order for the matroid to have rank two, there must be at least two parallel classes of non-loop elements.

Following~\cite{Ban+Pen+vdP:num-matroids}, we introduce the notion of a local cover, which is an object certifying that certain subsets are infeasible, enabling us to satisfy the second requirement of the proof. More precisely, given an even delta-matroid $D=(E,\Ff)$ a \emph{local cover} at $X$, for some subset $X$ of $E$, is a partition of $E \cup z$, where $z$ is an arbitrarily chosen element that is not in $E$. Let $x$ be the vertex of $R_{|E|}$ corresponding to $X$. If $X$ is infeasible with even size, then the local cover at $X$ certifies which of the subsets of $E$ corresponding to vertices in $N(x)$ are infeasible as follows. If the partition has strictly fewer than three blocks, then every subset of the form $X\sd \{a,b\}$ is infeasible.
Otherwise interpret this partition as a matroid $M$ on $E$ with rank two, as described above. A set $X \sd \{a,b\}$ is infeasible if and only if $\{a,b\}$ is not a basis of $M$.
It is clear that for any infeasible set $X$ with even size, one may construct a local cover at $X$ certifying which sets of the form $X \sd \{a,b\}$ are infeasible, in the way we have just described.

\begin{thm}
The number of even delta-matroids $e_n$ on $n$ elements satisfies
\[
\log\log e_n \leq n - \log n + O(\log\log n)
\]
\end{thm}

\begin{proof}
We first count the number of even delta-matroids with ground set $[n]$ such that every feasible set has even size, following the encoding procedure of Bansal, Pendavingh, and van der Pol \cite{Ban+Pen+vdP:num-matroids}.
Let $D$ be such a delta-matroid and let $L$ be the set of its infeasible sets having even size.
Recall that each component of $R_n$ is regular with degree $d=\binom n2$ and the adjacency matrix of a component of $R_n$ has smallest eigenvalue $-\lambda$ equal to ${-\big\lfloor\frac n2\big\rfloor}$.

To specify $L$, we first run the encoding procedure from~\cite{Ban+Pen+vdP:num-matroids} described above
with $\alpha =  \frac{\lambda}{d+\lambda}$ to obtain subsets $S$ and $A$ of the vertices of one component of $R_n$ such that $S \subseteq L \subseteq S \cup N(S) \cup A$,
$|S| \leq \big\lceil \frac{\ln(d+1)}{d+\lambda}N\big\rceil$ and $|A| \leq  \frac{\lambda}{d+\lambda}2^{n-1}$. Let $\sigma = \frac{\ln(d+1)}{d+\lambda}$.

We have
\[
\alpha =
\begin{cases}
	\mbox{ } \frac{1}{n} &\text{ if $n$ is even,} \\
	\frac{1}{n+1} &\text{ if $n$ is odd}
\end{cases}
\quad \text{ and } \quad
\sigma =
\begin{cases}
	\frac{2\ln\big(\binom n2+1\big)}{n^2} &\text{ if $n$ is even,} \\
	\frac{2\ln\big(\binom n2 + 1\big)}{n^2 - 1} &\text{ if $n$ is odd.}
\end{cases}
\]

Recall that $A$ is determined by $S$. All members of $L-A$ are contained in $S \cup N(S)$.
Thus in order to specify $L-A$, we require the set $S$ and
a local cover for each subset of $[n]$ corresponding to a member of $S$.
To specify $L \cap A$ we simply list the infeasible sets contained within $A$.

This bounds the number of even delta-matroids with ground set $[n]$ by twice the product of the number of ways of choosing $S$, the number of ways of choosing the corresponding sequence of local covers, one for each element of $S$, and the number of subsets of $A$. Let $B(n)$ denote the $n$th Bell number, that is, the number of partitions of a set of $n$ elements. A crude upper bound for $B(n)$ is given by $B(n) \leq n^n$.
We have
\[  e_n \leq 2 \sum_{i=0}^{\lceil\sigma 2^{n-1}\rceil} \Bigg(\binom{2^{n-1}} i (B(n+1))^i\Bigg) 2^{\frac{1}{n} 2^{n-1} }.\]
Let $\sigma' = \frac{1+\lceil\sigma 2^{n-1}\rceil}{2^{n-1}}$. Hence $\sigma \leq \sigma' \leq \sigma + \frac{1}{2^{n-2}}$.
Applying the inequality $\binom n k \leq ( \frac {ne}{k})^k$ and noting that $\sigma' \leq 1/2$ gives
\begin{align*}
 e_n &\leq \sigma'2^n \binom{2^{n-1}}{\sigma'2^{n-1}} (B(n+1))^{\sigma'2^{n-1}} 2^{\frac 1n 2^{n-1}}\\ &\leq  \sigma' 2^n \Big(\frac e {\sigma'}\Big)^{\sigma' 2^{n-1}}  (n+1)^{(n+1)\sigma'2^{n-1}} 2^{\frac 1n 2^{n-1}}.\end{align*}
Hence
\begin{align*}
\log e_n
&\leq \log \sigma' + n + \sigma' 2^{n-1} (\log e - \log \sigma') + (n+1)\sigma' 2^{n-1} \log (n+1) + \frac {2^{n-1}}n\\
& = 2^{n-1} \Big( \frac{\log \sigma'}{2^{n-1}} + \frac{n}{2^{n-1}} + \sigma' \log e - \sigma' \log \sigma' + (n+1) \sigma' \log(n+1) + \frac 1n\Big).
\end{align*}

We have
\[ \sigma' \leq \frac{2 \log ((n+1)^2)
}{n^2-1} + \frac{1}{2^{n-2}} \leq c_0\frac{\log(n+1)}{(n+1)^2},\]
and similarly
\[ \sigma' \geq c_1\frac{\log(n+1)}{(n+1)^2},\]
for some positive constants $c_0$ and $c_1$. Thus $\sigma' \log \sigma' \geq - c_2 \frac{(\log(n+1))^2}{(n+1)^2}$ for some positive constant $c_2$.
So
\begin{align*}
\log e_n
&\leq 2^{n-1} \Big( \frac{n}{2^{n-1}} + c_0\log e \frac{\log(n+1)}{(n+1)^2} + c_2 \frac{(\log(n+1))^2}{(n+1)^2}\\
 & \phantom{\leq} \ { } +  c_0 \frac{(\log(n+1))^2}{n+1} + \frac 1n\Big)\\
&\leq 2^{n-1} c_3 \frac{ (\log(n+1))^2}{n+1},
\end{align*}
for some positive constant $c_3$, as all the terms in the brackets in the previous line have order at most $\frac{(\log(n+1))^2}{n+1}$.
Finally we obtain
\[ \log \log e_n \leq n -\log n + O(\log \log n).
\] \qed
\end{proof}

\section*{Acknowledgements}
We thank Gordon Royle for independently verifying the numbers of labelled delta-matroids with up to six elements. We also thank the anonymous referees for their careful reading and useful comments.

\bibliographystyle{elsarticle-num}

\begin{thebibliography}{99}
\bibitem{Alon+Balogh+Morris+Samotij}
Noga Alon, J\'ozsef Balogh, Robert Morris and Wojciech Samotij.
\newblock Counting sum-free sets in Abelian groups.
\newblock \textit{Israel Journal of Mathematics}, \textbf{199} (2014) 309--344.

\bibitem{Ban+Pen+vdP:num-matroids}
Nikhil Bansal, Rudi~A. Pendavingh and Jorn~G. van~der Pol.
\newblock On the number of matroids.
\newblock \textit{Combinatorica}, \textbf{35} (2015) 253--277.

\bibitem{BCN}
J.~Bonin, C.~Chun and S.~D. Noble.
\newblock Delta-matroids as subsystems of sequences of Higgs lifts.
\newblock \textit{In preparation}.

\bibitem{Bou:Symm}
A.~Bouchet.
\newblock Greedy algorithm and symmetric matroids.
\newblock \textit{Mathematical Programming}, \textbf{38} (1987) 147–-159.

\bibitem{Bou:Map}
A.~Bouchet.
\newblock Maps and delta-matroids.
\newblock \textit{Discrete Mathematics}, \textbf{78} (1989) 59--71.

\bibitem{Bou+Duch}
A.~Bouchet and A.~Duchamp.
\newblock Representability of delta-matroids over $GF(2)$.
\newblock \textit{Linear Algebra and its Applications}, \textbf{146} (1991) 67–-78.

\bibitem{Bri+Hoog:1}
Robert Brijder and Hendrik J.~Hoogeboom.
\newblock Nullity and loop complementation for delta-matroids.
\newblock \textit{SIAM Journal of Discrete Mathematics}, \textbf{27} (2013) 492--506.

\bibitem{Bri+Hoog:2}
Robert Brijder and Hendrik J.~Hoogeboom.
\newblock Interlace polynomials for multimatroids and delta-matroids.
\newblock \textit{European Journal of Combinatorics},  \textbf{40} (2014) 142--167.

\bibitem{Brouwer+Haemers:spectra}
Andries~E. Brouwer and Willem~H. Haemers.
\newblock \textit{Spectra of graphs}.
\newblock Universitext. Springer, New York, 2012.

\bibitem{CMNR1}
Carolyn Chun, Iain Moffatt, Steven Noble, Ralf Rueckriemen.
\newblock Matroids, delta-matroids and embedded graphs.
\newblock arXiv:1403.0920 (2014).

\bibitem{CMNR2}
Carolyn Chun, Iain Moffatt, Steven Noble, Ralf Rueckriemen.
\newblock On the interplay between embedded graphs and delta-matroids.
\newblock arXiv:1602.01306 (2016).

\bibitem{Graham+Sloane}
R.~L. Graham and N.~J. A. Sloane.
\newblock Lower bounds for constant weight codes.
\newblock \textit{IEEE Transactions on Information Theory},
\textbf{26} (1980) 37–-43.

\bibitem{Kleitman+Winston}
Daniel J. Kleitman and Kenneth J. Winston.
\newblock On the number of graphs without 4-cycles.
\newblock \textit{Discrete Mathematics},
\textbf{41} (1982) 167–-172.

\bibitem{Knuth}
Donald E. Knuth.
\newblock The asymptotic number of geometries.
\newblock \textit{Journal of Combinatorial Theory Series A},
\textbf{16} (1974) 398--400.

\bibitem{Kor+Sap:hypercube}
A.~D. Korshunov and A.~A. Sapozhenko.
\newblock The number of binary codes with distance 2. (In Russian.)
\newblock \textit{Problemy Kibernet}, \textbf{40} (1983) 111–-130.

\bibitem{Mayhew+Newman+Welsh+Whittle:asymptoticmatroid}
Dillon Mayhew, Mike Newman, Dominic Welsh and Geoff Whittle.
\newblock On the asymptotic proportion of connected matroids.
\newblock \textit{European Journal of Combinatorics},
\textbf{32} (2011) 882--890.

\bibitem{Oxley}
J.~G. Oxley.
\newblock \emph{Matroid theory} 2nd edition.
\newblock Oxford University Press, Oxford (2011).

\bibitem{Piff}
M.~J. Piff.
\newblock An upper bound for the number of matroids.
\newblock \textit{Journal of Combinatorial Theory Series B}, \textbf{14} (1973) 241–-245.

\bibitem{Piff+Welsh}
M.~J. Piff and D.~J. A. Welsh.
\newblock The number of combinatorial geometries.
\newblock \textit{Bulletin of the London Mathematical Society},
\textbf{3} (1971) 55-–56.

\bibitem{Royle}
G.~F. Royle.
\newblock \textit{Private communication}, (2016).

\bibitem{Tralid+Brijder+Hoogeboom}
Lorenzo Traldi, Robert Brijder and Hendrik J.~Hoogeboom.
\newblock The adjacency matroid of a graph.
\newblock \textit{The Electronic Journal of Combinatorics},
\textbf{20} (2013) \#P27.

\bibitem{Welsh}
D.~J. A. Welsh.
\newblock \textit{Matroid theory}.
\newblock Academic Press, London -- New York (1976).

\bibitem{Whitney}
Hassler Whitney.
\newblock On the abstract properties of linear dependence.
\newblock \textit{American Journal of Mathematics},
\textbf{57} (1935) 509--533.


\end{thebibliography}

\end{document}